\documentclass{article}
\usepackage[utf8]{inputenc}
\usepackage{amsthm}
\usepackage{amssymb}
\usepackage{amsmath}
\usepackage{amsfonts}
\usepackage{cancel}
\usepackage{times}
\usepackage{bbm}
\usepackage{cite}
\usepackage{color, soul}
\usepackage{hyperref}
\hypersetup{
     colorlinks   = true,
     citecolor    = black,
     linkcolor    = blue
}

\usepackage[normalem]{ulem}

\usepackage[paper=a4paper, top=2.5cm, bottom=2.5cm, left=2.5cm, right=2.5cm]{geometry}

\def\qed{\hfill $\vcenter{\hrule height .3mm
\hbox {\vrule width .3mm height 2.1mm \kern 2mm \vrule width .3mm
height 2.1mm} \hrule height .3mm}$ \bigskip}

\def \Sph{\mathbb{S}^{n-1}}
\def \RR {\mathbb R}

\def \EE {\mathbb E}

\def \PP {\mathbb P}

\def \c\EE {\mathcal \EE}

\def \R {\mathbb{R}}

\def \\EEPS {\delta}

\def \Ric {\operatorname{Ric}}
\def \Riem {\operatorname{Riem}}
\def \para {//}
\def \gradsph {\nabla_S}
\def \Hesssph{\nabla_S^2}

\newtheorem{theorem}{Theorem}
\newtheorem{lemma}{Lemma}

\newtheorem{conjecture}{Conjecture}
\newtheorem{proposition}{Proposition}

\theoremstyle{definition}
\newtheorem{definition}[theorem]{Definition}
\theoremstyle{remark}
\newtheorem{remark}{Remark}

\newcommand{\rev}[1]{{\color{black} #1}}

\long\def\symbolfootnotetext[#1]#2{\begingroup
\def\thefootnote{\fnsymbol{footnote}}\footnotetext[#1]{#2}\endgroup}

\title{Transportation onto log-Lipschitz perturbations}
\date{}

\author{Max Fathi, Dan Mikulincer, and Yair Shenfeld}
\author{Max Fathi\thanks{Universit\'e Paris Cité and Sorbonne Universit\'e, CNRS, Laboratoire Jacques Louis Lions \&
		Laboratoire de Probabilit\'es Statistique et Mod\'elisation, F-75013 Paris, France. \\
Ecole Normale Sup´erieure, Universit´e PSL, D´epartement de Math´ematiques et applications, F-75005, Paris,
France\\
Institut Universitaire de France (IUF). Email address: \texttt{mfathi@lpsm.paris}}
	~~and Dan Mikulincer\thanks{Department of Mathematics, Massachusetts Institute of Technology, Cambridge, MA, USA. Email address: \texttt{danmiku@mit.edu}}
	~~and Yair Shenfeld\thanks{Division of Applied Mathematics, Brown University, Providence, RI, USA. Email address: \texttt{Yair\_Shenfeld@Brown.edu}} }
\begin{document}

\maketitle

\begin{abstract}
    We establish sufficient conditions for the existence of globally Lipschitz transport maps between probability measures and their log-Lipschitz perturbations, with dimension-free bounds. Our results include Gaussian measures on Euclidean spaces and uniform measures on spheres as source measures. More generally, we prove results for source measures on manifolds satisfying strong curvature assumptions. These seem to be the first examples of dimension-free Lipschitz transport maps in non-Euclidean settings, which are moreover sharp on the sphere.
    We also present some applications to functional inequalities, including a new dimension-free Gaussian isoperimetric inequality for log-Lipschitz perturbations of the standard Gaussian measure. Our proofs are based on the Langevin flow construction of transport maps of Kim and Milman.
\end{abstract}

\section{Introduction} 

This work concerns the problem of transporting a probability measure $d\mu=e^{-V}d\operatorname{Vol}$ on a Riemannian manifold $(M,g)$  onto a measure $d\nu=e^{-(V+W)}d\operatorname{Vol}$ which is an $L$-log-Lipschitz perturbation of $\mu$, i.e., $|\nabla W|\le L$. We will show that the \emph{Langevin transport map} (also known as the heat flow transport map of Kim and Milman) between $\mu$ and $\nu$ is Lipschitz in various Euclidean and manifold settings. This construction of transport maps was introduced in \cite{KiMi12}. Our main results are Theorem \ref{thm:Euclidean} in the Euclidean setting and Theorem \ref{thm:manifold_intro} in the Riemannian setting (with sharp results in the special case of the sphere, Theorem \ref{thm_sphere}). We shall also discuss applications to functional inequalities and optimal transport.

The question of proving Lipschitz bounds for transport maps goes back to the seminal work of Caffarelli \cite{caffarelli2000monotonicity}, who proved that the quadratic optimal transport map (or Brenier map) from a standard Gaussian measure onto a uniformly log-concave measure on $\R^n$ is globally Lipschitz, with a dimension-free bound (see \cite{FGP, CP} for alternative proofs based on entropic optimal transport and \cite{kolesnikov_sobolev} for Sobolev bounds). The existence of such maps allows to transfer various functional, isoperimetric, and concentration inequalities from the source measure to the target measure \cite{cordero2002some, caffarelli2000monotonicity, Har99}. For example, Caffarelli's result immediately recovers the classical results of Bakry and \'Emery on sharp functional inequalities for uniformly log-concave measures. Moreover, E. Milman \cite{milman2018spectral} showed that such Lipschitz estimates imply bounds on higher eigenvalues of certain differential operators, for which no non-transport proofs are known at this time. What is crucial for many applications, such as the correlation inequalities \cite{Har99} and quantitative central limit theorems \cite{mikulincer_clt}, is to ensure that the Lipschitz estimates are dimension-free.

Several extensions of Caffarelli's theorem on the Brenier maps have been proven since then, such as \cite{colombo_lipschitz}, which showed that the optimal transport map from a Gaussian measure onto any log-compactly-supported perturbation is globally Lipschitz. Beyond these extensions not much is known about the Lipschitz properties of the optimal transport map, and there are many remaining questions \cite{kolesnikov2011mass}. However, for most applications, it is not particularly important that the map is optimal and any globally Lipschitz map will suffice. An emerging line of research has focused on the construction and analysis of Lipschitz transport maps beyond the setting of optimal transport. Our starting point is the paper \cite{KiMi12}, in which Y. Kim and E. Milman introduced a new construction of transport maps and recovered, as well as extended, Caffarelli's result. The construction is based on a time reversal of an overdamped Langevin (or drift-diffusion) SDE, and we shall call it the Langevin transport map in the sequel. As shown by Tanana \cite{tanana_comparison}, this construction does not coincide with the Brenier map in general (although they do coincide in dimension one), see also \cite{LaSa22}. It is on this construction that the present work is based.

More recently, there have been several works investigating Lipschitz properties of the Langevin transport map in the Euclidean setting when the target measure satisfies certain convexity conditions \cite{MS21, neeman2022lipschitz, klartag_Putterman, shenfeld2022}. Our focus here is on a different class of target measures, with a first-order condition on the target rather than a second-order condition (but still using strong convexity assumptions on the source measure). As a general motivation for relaxing regularity assumptions, we mention that reversing diffusion processes to construct transport maps has recently gained a lot of traction in the machine learning community \cite{HJA20,BTHD21}, where such constructions are used to generate samples from unknown distributions. 

The study of the Langevin transport map led to new results that were previously unattainable for the optimal transport map. However, all results mentioned above are limited to measures on Euclidean spaces, while the question of Lipschitz transport maps is also interesting for non-Euclidean geometries. Indeed, for Riemannian manifolds there are some motivating questions of E. Milman \cite{milman2018spectral} and of Beck and Jerison \cite{beck_friedland}. In light of this, an additional goal in this work is to establish sufficient conditions on weighted Riemannian manifolds to ensure the existence of Lipschitz transport maps. While the existence of a Lipschitz optimal transport map can be established in some manifold settings \cite[Theorem 3.1]{MR2648684}, the bounds are dimension-dependent and the necessary assumptions are highly restrictive. In contrast, under the first-order condition described above, our results yield, for the first time, a construction of globally Lipschitz transport maps with explicit dimension-free bounds in a general manifold setting. 

\subsection*{Main results}
 As explained above, our results concern Lipschitz estimates for the Langevin transport map between a source measure (on $\R^n$ or on a Riemannian manifold) and a target measure on the same space, whose density with respect to the original measure is log-Lipschitz.
 \paragraph{Euclidean spaces:} Our main result in the Euclidean setting, which shall be proved in Section \ref{sec:Euclidean}, is

\begin{theorem}
\label{thm:Euclidean}
Let $\kappa>0$ and suppose that $V:\R^n\to \R$ is a thrice-differentiable function such that
\[
\nabla^2 V(x) 	\succeq \kappa \mathrm{I}\quad \forall ~x\in \R^n \quad\text{and} \quad |\nabla^3V(x)(u,u)|\le K \quad\forall ~u\in \Sph.
\]
Let $d\mu=e^{-V}dx$ and $d\nu=e^{-(V+W)}dx$ with \[
|\nabla W(x)|\le L \quad \forall~ x\in \R^n.
\]
Then, the Langevin transport map between $\mu$ and $\nu$ is Lipschitz with constant 
$\exp\left(10\left[\frac{L}{\sqrt{\kappa}}+\frac{L^2}{\kappa}+\frac{LK}{\kappa^2}\right]\right)$.
\end{theorem}

\begin{remark}
The log-Lipschitz assumption implies similar results for measures $\nu$ whose relative density with respect to $\mu$ is Lipschitz and strictly  bounded from below. Indeed, if  $|\nabla e^{-W}|\le C$ and $e^{-W}>c>0$, then $|\nabla W|\le L:=\frac{C}{c}$. 
\end{remark}
The dependence on $L$ in Theorem \ref{thm:Euclidean} is sharp as can be evidenced by the standard Gaussian. Indeed, if $\mu$ is the standard Gaussian measure on $\R$, then $\kappa=1$ and $K=0$. We now choose \rev{$W(x)=-L|x|+\log Z$}, where $Z$ is a normalizing constant, and suppose that there is an $M$-Lipschitz transport map from $\mu$ to $\nu$. Then, as in \cite{cordero2002some, kolesnikov2011mass}, the Gaussian isoperimetric inequality can be transported to $\nu$, up to a factor of $\frac{1}{M}$. More precisely, given a set $A\subset \R$ satisfying $\nu(A)=\frac{1}{2}$, we have that $\nu^+(A)\ge \frac{1}{\sqrt{2\pi}M}$ where $\nu^+(A):=\lim_{\varepsilon\downarrow 0}\frac{\nu(A^{\varepsilon})-\nu(A)}{\varepsilon}$ with $A^{\varepsilon}$ standing for the $\varepsilon$-neighborhood of $A$. We now take $A=[0,\infty)$ and note that $\nu(A)=\frac{1}{2}$, by symmetry, and that $\nu^+(A)=\frac{d\nu}{dx}(0)$. We have
\[
Z=\int e^{L|x|}d\mu \ge \int e^{-Lx}d\mu=e^{\frac{L^2}{2}}
\]
so 
\[
\frac{d\nu}{dx}(0) = e^{-W(0)}\frac{d\mu}{dx}(0) =\frac{1}{Z}\frac{1}{\sqrt{2\pi}} \leq \frac{e^{-\frac{L^2}{2}}}{\sqrt{2\pi}}.
\]
It follows that $M\ge e^{\frac{L^2}{2}}$. 

\paragraph{Riemannian manifolds:} We now discuss the manifold setting in which our first result applies to the uniform measure on the round sphere $\Sph$, with radius scaled as $\sqrt{n-2}$. In this case, we obtain a sharp estimate, similar to Theorem \ref{thm:Euclidean}.
\begin{theorem} \label{thm_sphere}
Let $n \geq 3$. Then, the Langevin transport map from the uniform measure on $\Sph(\sqrt{n-2})$ onto a measure with $L$-log-Lipschitz density is Lipschitz with constant \rev{$\exp\left(35(L+ L^2)\right)$}.
\end{theorem} 

The scaling in $\sqrt{n}$ of the radius is essentially sharp for dimension-free behavior, as well as the behavior in $\exp(L + L^2)$ of the Lipschitz constant,  since if we let $n$ go to infinity, the distribution of a single coordinate converges to a standard Gaussian measure, so we should recover the same type of estimates as in the Gaussian case in the limit, which is indeed the case with our estimate. When compared to Theorem \ref{thm:Euclidean}, one may see that the convexity term $\kappa$ is replaced by the Ricci curvature of the sphere, which is of order $1$ when the diameter scales like $\sqrt{n}$. The other difference is the existence of the regularity parameter $K$, which does not appear in Theorem \ref{thm_sphere}. Indeed, the source measure is uniform and its density is constant. 
\newline

In Section \ref{sec:genmanifold} we shall consider a more general setting of weighted Riemannian manifolds and our bounds will be given both in terms of the regularity of the density and the curvature of the manifold, as in Theorem \ref{thm_sphere}. In particular, our results shall apply to manifolds that satisfy the Bakry-\'Emery curvature-dimension condition, see \cite[Chapter 1.6]{bakry2013analysis}, as well as some extra conditions. 
We state here an informal version of the theorem and defer the exact statement and definitions to Theorem \ref{thm:manifold} in Section \ref{sec:genmanifold}. 
\begin{theorem}
\label{thm:manifold_intro}
Let $(M,g,\mu)$ be a weighted Riemannian manifold with $\mu = e^{-V}d\mathrm{Vol},$ and let $\nu = e^{-W}d\mu$, with $W$ an $L$-Lipschitz function. Then, under appropriate assumptions on the curvature of $M$ and on $V$, 
the Langevin transport map from $\mu$ onto $\nu$ is Lipschitz with constant $e^{e^{cL^2}}$ for large $L$, where $c$ can be made explicit, and depends on the curvature and on $V$, but not on the dimension. 
\end{theorem}
The estimate of Theorem \ref{thm:manifold_intro} is still dimension-free, although significantly worse than what we derived in Theorem \ref{thm:Euclidean} and Theorem \ref{thm_sphere}. We believe that the estimate  of Theorem \ref{thm:manifold_intro} is suboptimal and that the double exponential can be omitted, as in the sphere.

As far as we know, in the Riemannian setting, the above-mentioned results are the first results about Lipschitz transport maps with a dimension-free behavior (under appropriate scaling). The existence of such maps paves the way to several applications of interest. In Section \ref{sect_applications}, we prove new functional inequalities, both in the Euclidean and Riemannian settings, and we shall also discuss open problems on optimal transport maps.

\paragraph{The inverse map:} In all of the examples above, we have taken a well-conditioned weighted manifold that satisfies some desirable combination of convexity and curvature assumptions. In these cases, we showed that log-Lipschitz perturbations can be realized as push-forward by Lipschitz mappings. We also investigate the analogous question in the reverse direction. Namely, we show that there exists a Lipschitz transport map from the perturbation, which is often not as well-behaved, back to the source measure.

\begin{theorem}
\label{thm:meta_reverse_intro}
In all of the above settings, the inverses of the Langevin transport maps are Lipschitz with dimension-free constants. 
\end{theorem}
The exact value of the Lipschitz constants turn out to be comparable to the Lipschitz constants of the Langevin transport maps. In Section \ref{sec:rev_lip} we give the exact dependence of the Lipschitz constants on all parameters. Theorem \ref{thm:meta_reverse_intro} has an additional interesting consequence: since maps can be composed, the theorem implies the existence of Lipschitz maps from one log-Lipschitz perturbation onto any other. 

\section*{Acknowledgments} We gratefully acknowledge the contributions of Joe Neeman who was involved in the earlier stages of this work. In particular,  the example showing Theorem \ref{thm:Euclidean} is sharp is due to him. We also thank Patrick Cattiaux, Daniel Lacker, Michel Ledoux, and Lorenzo Schiavo for useful discussions. In addition, we thank the anonymous referee for their helpful comments.  

M.F. was supported by the Projects MESA (ANR-18-CE40-006) and EFI (ANR-17-CE40-0030) of the French National Research Agency (ANR). This material is based upon work supported by the National
Science Foundation under Award Number 2002022.

\noindent \textbf{Statement on AI usage.} ChatGPT 5.5 Pro was used to correct mathematical and grammatical typos in the v3 arXiv version of this paper. 

\section{Preliminaries} \label{sec:LT}

Before explaining the construction of the Langevin transport map, we introduce the Langevin dynamics on $\R^n$. Let $d\mu=e^{-V}dx$ be a probability measure on $\R^n$ with $V: \R^n\to\R$ twice-differentiable. The \emph{Langevin dynamics}, associated to $\mu$, starting at $x\in \R^n$ is the solution of the stochastic differential equation
\begin{align}
\label{eq:SDElangevin}
dX_t^x=-\nabla V(X_t^x)dt+\sqrt{2}d\omega_t,\quad X_0^x=x,
\end{align}
where $(\omega_t)$ is a standard Brownian motion in $\R^n$.  We denote by $(P_t)$ the Langevin semigroup,
$$
P_tf(x):= \EE [f(X_t^x)],
$$
for any $f:\R^n\to \R$ for which the above is defined. It is well known that under appropriate regularity assumptions $(P_t)$ is ergodic. Consequently, for any $x\in \R^n$, the law of $X_t^x$ converges weakly to $\mu$ as $t\to\infty$. 

The definition of the Langevin dynamics can readily be extended to Riemannian manifolds (see \cite{bakry2013analysis,Hsu} for in-depth discussions). Let $(M,g)$ be a complete $n$-dimensional Riemannian manifold, and let $O(M)$ be its orthonormal frame bundle. Let $(B_t)$ be a standard Brownian motion in $\R^n$ and let $(\Phi_t)$ be the horizontal Brownian motion on $O(M$) defined by
\[
d\Phi_t=\sum_{i=1}^nH^i(\Phi_t)\circ dB_t,
\]
where $\{H^i\}_{i\in [n]}$ are the horizontal lifts of the standard basis $\{e_i\}_{i\in [n]}$, and $\circ$ stands for the Stratonovich integral. We assume that $(\Phi_t)$ does not explode in  finite time. The process $(\omega_t)$ given by $\omega_t=\pi\Phi_t$, where $\pi:O(M)\to M$ is the canonical projection, is what we refer to as the Brownian motion on $M$.

Now, as in the Euclidean case, given  a probability measure $d\mu=e^{-V}d\mathrm{Vol}$ on $M$, where $d\mathrm{Vol}$ is the volume measure on $M$ and $V:M\to \R$ is twice-differentiable, we define the \emph{Langevin dynamics} starting at $x\in M$ as the solution of the stochastic differential equation
\begin{align}
\label{eq:SDE_Langevin_manifold}
dX_t^x=-\nabla_M V(X_t^x)dt+\sqrt{2}\circ d\omega_t,\quad X_0^x=x,
\end{align}
where $\nabla_M$ is the gradient in $M$. Note that the use of the Stratonovich integral instead of the It\^ o integral is immaterial in our setting because the coefficient of the Brownian motion is a constant. 

\subsection*{The Langevin transport map}
\label{subsubsec:tranport}
We now briefly explain the construction of the Langevin transport map. The reader is referred to \cite{MS21} for further details on the constructions introduced in this section. Let $d\nu=e^{-W}d\mu$ be a probability measure on $M$ and consider the Langevin dynamics, associated to $\mu$, and starting at $\nu$,
\begin{align}
\label{eq:SDElangevinnu}
X_t:=X_t^x \ \ \text{with} \ \ x\sim \nu.
\end{align}
The flow of probability measures $(\rho_t):=(\text{Law}(X_t))$ forms an interpolation between $\rho_0=\nu$ and $\rho_{\infty}=\mu$. In particular, it satisfies the transport equation
\[
\partial_t\rho_t=-\nabla_M \cdot (\rho_t\nabla_M\log P_t e^{-W}),
\]
with $\nabla_M$ standing for the gradient in $M$. We define a flow of diffeomorphisms $\{S_t\}_{t \geq 0}$ as a solution to the following integral curve equation
\begin{align}
\label{eq:Langevin_flow}
\partial_t S_t=-\nabla_M\log P_t e^{-W}(S_t),\quad S_0(x)=x ~~\forall x\in M.
\end{align}
It turns out that for any $t \geq 0$, $S_t$ transports $\nu$ to $\rho_t$. Setting $T_t:=S_t^{-1}$, and letting 
\[
T:=\lim_{t\to\infty} T_t,
\]
we see that $T$ transports $\mu$ to $\nu$. This is the \emph{Langevin transport map} between $\mu$ and $\nu$.\footnote{The existence of solutions to the above equations, and of the limits, is non-trivial in general. These issues are discussed in length in \cite{otto2000generalization, KiMi12, MS21} for the various settings we consider. In particular, the convexity and curvature assumptions we shall enforce in the coming proofs are enough to guarantee that all maps under consideration are well-defined.}
The particular form of the equation in \eqref{eq:Langevin_flow} is particularly amenable to establishing Lipschitz properties of $T$ from estimates of $\nabla_M^2\log P_te^{-W}$, where $\nabla_M^2$ is the Hessian in $M$,
since differentiating \eqref{eq:Langevin_flow} yields
\[
\partial_t \nabla_MS_t=-\nabla_M^2\log P_t e^{-W}(S_t)\nabla_MS_t,\quad \nabla_MS_0=\mathrm{I}.
\]
This is the content of the following Lemma. The proof is an adaption of \cite[Lemma 3]{MS21} and \cite[Lemma 3.2]{neeman2022lipschitz} which only consider Euclidean spaces.

\begin{proposition}
\label{prop:2nd_derivative_to_Lipschitz}
If for all $t\ge 0$,
\[
\nabla_M^2\log P_te^{-W}\preceq \theta_t g,
\]
then $T$ is $\exp\left(\int_0^{\infty}\theta_tdt\right)$-Lipschitz.
\end{proposition}
\begin{proof}
Consider $F_t : M \longrightarrow \R$ a time-dependent family of smooth functions on $(M,g)$, such that for any $t > 0$ we have
$$\nabla_M^2 F_t \geq -\theta_t g.$$
As is classical (e.g., \cite[Proposition 16.2(iv')]{villani2008optimal}), this is equivalent to saying that for any $t > 0$, $x,y \in M$ and constant speed geodesic $\alpha : [0, 1] \longrightarrow M$ connecting $x$ to $y$, we have
$$\langle \dot{\alpha}(1), \nabla_M F_t(y) \rangle - \langle \dot{\alpha}(0), \nabla_M F_t(x)\rangle \geq -\theta_t d(x,y)^2.$$

Therefore, if $x_t$ and $y_t$ are time-dependent gradient flows of $F_t$ with different initial data, and $\alpha^t$ is a constant-speed geodesic connecting $x_t$ to $y_t$, we have (with $d$ standing for the distance induced by the Riemannian metric),
\begin{align*}
    \frac{d}{dt}d(x_t,y_t)^2 &= 2\left( \langle \dot{y}_t, \dot{\alpha}^t(1) \rangle-\langle \dot{x}_t, \dot{\alpha}^t(0)\rangle  \right) \\
    &= 2\left( \langle \nabla_M F(y_t), \dot{\alpha}^t(1) \rangle-\langle \nabla_M F(x_t),\dot{\alpha}^t(0)\rangle \right) \\
    &\geq -2\theta_t d(x_t,y_t)^2.
\end{align*}
The rest of the proof proceeds as in the proof of \cite[Lemma 3.2]{neeman2022lipschitz}, with the manifold setting making no difference.
\end{proof}

 As is clear from Proposition \ref{prop:2nd_derivative_to_Lipschitz}, our goal will be to estimate the Hessian along the Langevin dynamics. Such estimates have previously been studied in the literature, in particular by Elworthy and Li \cite{ElworthyLi}. We also mention \cite{thompson2020}, which contains a pedagogical exposition of second derivative computations and estimates in the manifold setting. For the general manifold setting, our results will be based on the recent work of Cheng, Thalmaier, and Wang \cite{cheng2021}. However, for the Euclidean and Sphere setting, our sharp results require new Hessian estimates, which we shall provide. All of the above methods are based on Bismut's integration by parts formula \cite{Bis81}. 

\section{Euclidean spaces}
\label{sec:Euclidean}
This section is devoted to the proof of Theorem \ref{thm:Euclidean}. For the rest of this section, we fix two probability measures $\mu = e^{-V(x)}dx$ and $\nu=e^{-W(x)}d\mu$ on $\RR^n$. We require that $\mu$, the source, satisfies the assumptions of Theorem \ref{thm:Euclidean}, 
\[
\nabla^2 V(x) 	\succeq \kappa \mathrm{I}\quad \forall ~x\in \R^n \quad\text{and} \quad |\nabla^3V(x)(u,u)|\le K \quad\forall ~u\in \Sph,
\]
for some $\kappa > 0$ and $K \geq 0$.
We also require that $\nu$, the target, is an $L$-log-Lipschitz perturbation of $\mu$,
\[
|\nabla W(x)|\le L \quad \forall~ x\in \R^n.
\]
Further, as in Section \ref{sec:LT}, we let $X_t$ stand for the Langevin dynamics, associated to $\mu$ with $X_0 \sim \nu$, and $P_t$ is its semigroup. As suggested by Proposition \ref{prop:2nd_derivative_to_Lipschitz} we shall require the following global bound of $\nabla^2\log P_t e^{-W(x)}$. 
\begin{proposition}
\label{prop:2nd_derivative_Euclidean}
For every $t\ge 0$ and $x \in \RR^n$,
\begin{align*}
 \nabla^2\log P_t e^{-W(x)}\preceq Le^{-\kappa t}\left[5L+\frac{5}{\sqrt{t}}+\frac{Kt}{2}\right] \mathrm{I}.
\end{align*}
\end{proposition}
The proof of Proposition \ref{prop:2nd_derivative_Euclidean} will be the main focus of this section. Before delving into the proof, let us show how Theorem \ref{thm:Euclidean} follows.
\begin{proof}[Proof of Theorem \ref{thm:Euclidean}]
Using 
\begin{align*}
\int_0^{\infty}e^{-\kappa t} dt=\frac{1}{\kappa}, \quad\quad\int_0^{\infty} \frac{e^{-\kappa t}}{\sqrt{t}}dt=\sqrt{\frac{\pi}{\kappa}}, \quad\quad \int_0^{\infty}  t e^{-\kappa t} dt=\frac{1}{\kappa^2},
\end{align*}
the combination of Proposition \ref{prop:2nd_derivative_to_Lipschitz} and Proposition \ref{prop:2nd_derivative_Euclidean} yields 
that $T$, the Langevin transport map, is Lipschitz with constant $\exp\left(\frac{5L^2}{\kappa}+\frac{5\sqrt{\pi}L}{\sqrt{\kappa}}+\frac{LK}{2\kappa^2}\right)\le \exp\left(10\left[\frac{L}{\sqrt{\kappa}}+\frac{L^2}{\kappa}+\frac{LK}{\kappa^2}\right]\right)$.
\end{proof}
\subsection{Stochastic representation of the Hessian}
In order to establish Lipschitz properties of the Langevin transport map our main task will be to derive global bounds on $\nabla^2\log P_te^{-W}$. Toward this goal, our main technical tool is a stochastic representation of $\nabla^2P_t$, known as \emph{Bismut's formula}. Before stating the formula, we first define the \emph{Jacobi flow} of the Langevin dynamics as $(\nabla X_t^x)_{t\geq 0}$, where $\nabla X_t^x:\R^n\to\R^n$ is the derivative of (the random function) $X^x_t$ with respect to the initial condition $x$,
\[
\nabla_u X_t^x:=\lim_{\varepsilon\downarrow 0}\frac{X_t^{x+\varepsilon u}-X_t^x}{\varepsilon}\in \R^n, \quad\forall u\in \R^n.
\]
The Jacobi flow $(\nabla X_t^x)_{t\geq 0}$ appears naturally when applying the chain rule in the differentiation of $ P_tf(x)$ with respect to $x$, formally,
\[
\nabla P_tf(x)=\nabla \EE[f(X_t^x)]=\EE[\nabla f(X_t^x)\nabla X_t^x].
\]
The process  $(\nabla X_t^x)$ satisfies the differential equation
\begin{align}
\label{eq:Jacobi_Euclidean}
\partial_t\nabla X_t^x=-\nabla^2V(X_t^x)\nabla X_t^x, \quad \nabla X_0^x=\mathrm{I},
\end{align}
which can be seen by differentiating \eqref{eq:SDElangevin} and noting that $(\omega_t)$ does not depend on $x$. 
Since we require Hessian estimates of $P_tf$, we also consider the \emph{second variation} of $(X_t^x)$, 
\[
\nabla_{u,v}^2X_t^x:=\lim_{\varepsilon\downarrow 0}\frac{\nabla_v X_t^{x+\varepsilon u}-\nabla_v X_t^x}{\varepsilon}\in \R^n, \quad \forall u,v\in \R^n
\]
(which is symmetric with respect to $u$ and $v$). The equation for the process $(\nabla^2X_t^x)$ can be obtained by differentiating \eqref{eq:Jacobi_Euclidean},
\begin{align}
\label{eq:Jacobi_derivative_Eucliean}
\partial_t\nabla_{u,v}^2 X_t^x=-\nabla^3V(X_t^x)(\nabla_u X_t^x,\nabla_v X_t^x)-\nabla^2V(X_t^x)\nabla_{u,v}^2 X_t^x, \quad \nabla^2 X_0^x=0,
\end{align}
where, for $x\in \R^n$ and $u,v\in \R^n$,
\[
[\nabla^3V(x)(u,v)]_i=\sum_{j,k=1}^n\partial^3_{ijk}V(x)u_jv_k=\nabla_u\nabla_v(\nabla V)_i(x)\quad \forall ~i\in [n].
\]

With these definitions, and the following notation,
$$\nabla_vf(x):= \langle f(x), v\rangle \quad\text{ and } \quad\nabla^2_{v,u}f(x) = \langle v,\nabla^2f(x)u\rangle$$
(the definition is extended naturally for vector-valued functions), the stochastic formula now reads:
\begin{lemma}{\textnormal{(\cite[p. 8 in arXiv version]{ElworthyLi})}}
\label{lem:Bismut}
$~$

For any $x\in \R^n$, $u,v\in \R^n$, and $f:\R^n\to \R$ twice-differentiable,
\begin{align*}
\nabla_{u,v}^2P_t f(x)&=\frac{1}{t\sqrt{2}}\EE\left[\nabla_v [f(X_t^x)] M_t^{x,u}\right]+\frac{1}{t}\int_0^t \EE\left[\langle \nabla P_{t-s}f(X_s^x),\nabla_{u,v}^2X_s^x\rangle\right]ds
\end{align*}
where 
\[
M_t^{x,u}:=\int_0^t\langle \nabla_u X_s^x, d\omega_s\rangle\quad x\in \R^n~, u\in \R^n.
\]
\end{lemma}
In our setting, $f=e^{-W}$ and we need to bound the various quantities appearing in  Lemma \ref{lem:Bismut}. The first derivative of $P_te^{-W}$ can be controlled by  the log-Lipschitz assumption. The first variation $\nabla X^x_t$ is controlled by the $\kappa$-log-concavity assumption, and the second variation $\nabla^2 X^x_t$ is controlled by the $\kappa$-log-concavity assumption and the bound on the third derivative of $V$.
\subsection{Estimates}
The purpose of this section is to provide the requisite estimates for the proof of Proposition \ref{prop:2nd_derivative_Euclidean}.
We start with the first variation.

\begin{lemma}{\textnormal{(First variation estimate)}}
\label{lem:1stderv}
$~$

Suppose that $\nabla^2 V(x) 	\succeq \kappa \mathrm{I}$ for every $x\in \R^n$. Then, a.s., for any fixed $v\in S^{n-1}$,
\[
|\nabla_vX_t^x|\le e^{-\kappa t}\quad \forall ~t\ge 0.
\]
Consequently, for any differentiable $f:\R^n\to \R$,
\[
|\nabla [f(X_t^x)]|\le e^{-\kappa t} |\nabla f(X_t^x)|.
\]
\end{lemma}
\begin{proof}
By the Cauchy-Schwarz inequality, for any $v\in \R^n$,
\[
|\nabla_v[f(X_t^x)]|=|\langle\nabla f(X_t^x),\nabla_vX_t^x\rangle|\le |\nabla f(X_t^x)||\nabla_vX_t^x|.
\]
Hence, taking the supremum over $v\in \Sph$ it suffices to show that $|\nabla_vX_t^x|\le e^{-\kappa t}$. To see the latter bound fix $v\in \Sph$ and let $\beta_v:\R_{\ge 0}\to \R_{\ge 0}$ be given by $\beta_v(t):=|\nabla_vX_t^x|$. Then, by \eqref{eq:Jacobi_Euclidean} and the convexity assumption on $V$,
\begin{align*}
\partial_t \beta_v(t)=\frac{\langle\nabla_vX_t^x),\partial_t\nabla_vX_t^x\rangle}{\beta_v(t)}=-\frac{\langle\nabla_vX_t^x,\nabla^2V(X_t^x)\nabla _vX_t^x\rangle}{\beta_v(t)}\le -\kappa\frac{|\nabla _vX_t^x|^2}{\beta_v(t)}=-\kappa \beta_v(t).
\end{align*}
Since $\beta_v(0)=1$, Gr\"onwall's inequality yields $|\nabla_vX_t^x|\le e^{-\kappa t}$.
\end{proof}

We now move to the second variation.

\begin{lemma}{\textnormal{(Second variation estimate)}}
\label{lem:2stderv}
$~$

Suppose that 
\[
\nabla^2 V(x) 	\succeq \kappa \mathrm{I}\quad \forall ~x\in \R^n \quad\text{and} \quad |\nabla^3V(x)(u,u)|\le K \quad\forall ~u\in \Sph.
\]
Then, a.s., for any $t\ge 0$ and $v\in \Sph$,
\[
|\nabla_{v,v}^2X_t^x|\le K t e^{-\kappa t}.
\]
\end{lemma}

\begin{proof}
Fix $v\in \Sph$ and let $\beta_v:\R_{\ge 0}\to \R_{\ge 0}$ be given by $\beta_v(t):=|\nabla_{v,v}^2X_t^x|$. By \eqref{eq:Jacobi_derivative_Eucliean} and Lemma \ref{lem:1stderv},
\begin{align*}
\partial_t\beta_v(t)&=\frac{\langle\nabla_{v,v}^2 X_t^x,\partial_t\nabla_{v,v}^2 X_t^x\rangle}{\beta_v(t)}=\frac{\langle\nabla_{v,v}^2 X_t^x,-\nabla^3V(X_t^x)(\nabla_v X_t^x,\nabla_v X_t^x)-\nabla^2V(X_t^x)\nabla_{v,v}^2 X_t^x\rangle}{\beta_v(t)}\\
&\le K \frac{|\nabla_{v,v}^2 X_t^x|\rev{|\nabla_v X_t^x|^2}}{\beta_v(t)}-\kappa \frac{|\nabla_{v,v}^2 X_t^x|^2}{\beta_v(t)}=K|\nabla_v X_t^x|  -\kappa \beta_v(t)\le Ke^{-\kappa t} -\kappa \beta_v(t).
\end{align*}
The solution to the differential equation 
\[
\partial_t\xi (t)=Ke^{-\kappa t} -\kappa \xi (t),\quad  \xi(0)=0
\]
is $\xi(t)=Kte^{-\kappa t}$ so by Gr\"onwall's inequality
\[
|\nabla_{v,v}^2X_t^x|=\beta_v(t)\le K t e^{-\kappa t}.
\]
\end{proof}

We will also use a reverse H\"older inequality for log-Lipschitz functions under a $\kappa$-log-concavity assumption: 
\begin{lemma}
\label{lem:reverseHollder}
Let $d\mu=e^{-V}dx$ and suppose that $\nabla^2 V(x) 	\succeq \kappa \mathrm{I}$ for every $x\in \R^n$. Then,  for every $L$-log-Lipschitz function $f:\R^n\to \R$,
\[
P_t(f^2)\le \exp\left(L^2\frac{1-e^{-2\kappa t}}{\kappa}\right)(P_tf)^2.
\]
\end{lemma}
\begin{proof}
It is well-known (e.g., \cite[Eq. (1.3)]{milman2020reverse}) that if a measure $\eta$ satisfies a log-Sobolev inequality with constant $C$ then it satisfies a reverse H\"older inequality for  $L$-log-Lipschitz functions,
\[
\|f\|_{L^p(\eta)}\le \exp\left(\frac{C}{2}L^2(p-q)\right)\|f\|_{L^q(\eta)}\quad \forall ~0\le q<p.
\]
Since  the measure $\eta:=P_t^*\delta_x$ satisfies a log-Sobolev inequality with constant $C=\frac{1-e^{-2\kappa t}}{2\kappa}$ \cite[Eq. (5.5.5)]{bakry2013analysis}, and since
\[
P_t(f^2)(x)=\|f\|_{L^2(\eta)}^2\quad\text{and}\quad (P_tf)^2=\|f\|_{L^1(\eta)}^2,
\]
we get
\[
P_t(f^2)\le  \exp\left(\frac{1}{\rev{2}}L^2\frac{1-e^{-2\kappa t}}{\kappa}\right)(P_tf)^2\le \exp\left(L^2\frac{1-e^{-2\kappa t}}{\kappa}\right)(P_tf)^2.
\]
\end{proof}
\begin{remark} \label{rmk:cdk}
The statement, and proof, of Lemma \ref{lem:reverseHollder} can be extended verbatim to the manifold setting by requiring that the measure satisfies the so-called curvature-dimension condition CD$(\kappa, \infty)$. In particular, the statement is for the uniform measure on the unit sphere $\Sph$, with $\kappa$ replaced by $n-2$.
\end{remark}
We shall also make use of the following classical lemma about martingales: 

\begin{lemma} \label{lem_deviation-mart}
    Let $M_t$ be a continuous martingale, with $M_0 = 0$, and whose quadratic variation is almost surely bounded, that is, $\langle M\rangle_t \leq \varphi(t)$ for all $t \geq 0$. Then
$$\PP[|M_t| \geq \delta ] \leq 2\exp\left(-\frac{\delta^2}{2\varphi(t)}\right).$$

\end{lemma}

\begin{proof}
From Novikov's condition, for any $\sigma \in \R$ we have
$$\EE[\exp(\sigma M_t - \sigma^2\langle M\rangle_t/2)] = 1,$$ so
$$\EE[\exp(\sigma M_t )] \leq \exp(\sigma^2 \varphi(t)/2).$$
We then apply the Chernoff bound, 
$$\PP[M_t \geq \delta] \leq \inf_{\sigma \geq 0} \exp(-\sigma \delta)\EE[\exp(\sigma M_t)]$$
to conclude. 
\end{proof}

\subsection{Proof of Proposition \ref{prop:2nd_derivative_Euclidean}}

By Lemma \ref{lem:Bismut}, 
\begin{align*}
\nabla_{u,v}^2P_t f(x)&=\frac{1}{t\sqrt{2}}\EE\left[\nabla_v [f(X_t^x)] M_t^{x,u}\right]+\frac{1}{t}\int_0^t \EE\left[(\nabla P_{t-s}f(X_s^x))^{\top}\nabla_{u,v}^2X_s^x\right]ds,
\end{align*}
and we will apply this identity with $f=e^{-W}$. We start by analyzing the first term. 
\begin{lemma}
\label{lem:martterm}
For any $v\in \Sph$,
\begin{align*}
\frac{1}{t\sqrt{2}}\EE\left[\nabla_v [e^{-W(X_t^x)}] M_t^{x,u}\right]\le 5e^{-\kappa t}\left[L^2+\frac{L}{\sqrt{t}}\right]P_te^{-W(x)}.
\end{align*}
\end{lemma}
\begin{proof}
For any $\delta\ge 0$ and $f$ differentiable,
\begin{align*}
|\EE\left[\nabla_v [f(X_t^x)] M_t^{x,u}\right]|&\le \EE\left[|\nabla_v [f(X_t^x)]| |M_t^{x,u}|\right]\\
&=\EE\left[1_{ |M_t^{x,u}|< \delta}|\nabla_v [f(X_t^x)]| |M_t^{x,u}|\right]+\EE\left[1_{ |M_t^{x,u}|\ge\delta}|\nabla_v [f(X_t^x)]| |M_t^{x,u}|\right]\\
&\le \delta \EE\left[|\nabla_v [f(X_t^x)]| \right]+\EE\left[|\nabla_v [f(X_t^x)]|^2\right]^{1/2}\EE\left[1_{ |M_t^{x,u}|\ge\delta} |M_t^{x,u}|^2\right]^{1/2}\\
&\le  \delta \EE\left[|\nabla_v [f(X_t^x)]| \right]+\EE\left[|\nabla_v [f(X_t^x)]|^2\right]^{1/2}\EE\left[|M_t^{x,u}|^4\right]^{1/4}\PP[ |M_t^{x,u}|\ge\delta]^{1/4}.
\end{align*}
We will now analyze the terms above one-by-one. For $v\in \Sph$, Lemma \ref{lem:1stderv} yields 
\begin{align*}
|\nabla_v [f(X_t^x)]|\le |\nabla [f(X_t^x)]| \le  e^{-\kappa t} |\nabla f(X_t^x)|= e^{-\kappa t}|f(X_t^x)| |\nabla \log f(X_t^x)|,
\end{align*}
so taking  $f=e^{-W}$ yields
\begin{align}
\label{eq:Jacobi_estimate}
|\nabla_v [e^{-W(X_t^x)}]|\le  Le^{-\kappa t}e^{-W(X_t^x)}.
\end{align}
It follows that 
\begin{align}
\label{eq:estimate1}
\EE[|\nabla_v [e^{-W(X_t^x)}]|]\le L e^{-\kappa t}P_te^{-W(x)}.
\end{align}
For the second term, \eqref{eq:Jacobi_estimate} gives
\[
\EE[|\nabla_v [e^{-W(X_t^x)}]|^2]^{1/2}\le L e^{-\kappa t}\left( P_t(e^{-W(x)})^2\right)^{1/2},
\]
so applying Lemma \ref{lem:reverseHollder} yields
\begin{align}
\label{eq:estimate2}
\EE[|\nabla_v [e^{-W(X_t^x)}]|^2]^{1/2}\le  L \exp\left(L^2\frac{1-e^{-2\kappa t}}{2\kappa}-\kappa t\right)P_te^{-W(x)}.
\end{align}
Next we turn to analyze the martingale $(M_t)$. By Lemma \ref{lem:1stderv}, $|\nabla_uX_t^x|\le e^{-\kappa t}$ a.s. for any $u\in \Sph$ so the quadratic variation of $(M_t)$ satisfies
\begin{align}
\label{eq:quadratic_variation}
\langle M^{x,u}\rangle_t=\int_0^t |\nabla_uX_s^x|^2 ds\le \frac{1-e^{-2\kappa t}}{2\kappa}.
\end{align}
On the other hand, the Burkholder-Davis-Gundy inequalities  yield that, for any $0<p<\infty$,
\begin{align}
\label{eq:BDG}
\EE[|M_t^{x,u}|^p]\le C_p\EE[\langle M^{x,u}\rangle_t^{p/2}],
\end{align}
for a constant $C_p$ depending only on $p$. 
Combining \eqref{eq:quadratic_variation} and \eqref{eq:BDG}, with $p=4$, and using $C_4\le 360$ (see \cite[proof of Proposition 3.26]{KaratzasShreve}), yields
\begin{align}
\label{eq:estimate3}
\EE[|M_t^{x,u}|^4]^{1/4}\le 5\sqrt{\frac{1-e^{-2\kappa t}}{2\kappa}}.
\end{align}
Moreover, applying Lemma \ref{lem_deviation-mart}, we have
\begin{equation*}
\PP[ |M_t^{x,u}|\ge\delta] \leq 2\exp\left(-\delta^2\frac{\kappa}{1-e^{-2\kappa t}}\right).
\end{equation*}
It follows that
\begin{align}
\label{eq:estimate4}
\PP[ |M_t^{x,u}|\ge\delta]^{1/4}\le 2^{1/4}\exp\left(-\frac{\delta^2}{4}\frac{\kappa}{1-e^{-2\kappa t}}\right).
\end{align}
Combining \eqref{eq:estimate1}, \eqref{eq:estimate2}, \eqref{eq:estimate3}, \eqref{eq:estimate4} yields
\begin{align*}
\frac{1}{t\sqrt{2}}&|\EE\left[\nabla_v [f(X_t^x)] M_t^{x,u}\right]|\\
&\le \frac{1}{t\sqrt{2}}\left[\delta L e^{-\kappa t}+ L \exp\left(L^2\frac{1-e^{-2\kappa t}}{2\kappa}-\kappa t\right)5\sqrt{\frac{1-e^{-2\kappa t}}{2\kappa}}2^{1/4}\exp\left(-\frac{\delta^2}{4}\frac{\kappa}{1-e^{-2\kappa t}}\right)\right]P_te^{-W(x)}\\
&= \frac{Le^{-\kappa t}}{t\sqrt{2}}\left[\delta  + 5\cdot 2^{1/4}\sqrt{\frac{1-e^{-2\kappa t}}{2\kappa}}\exp\left(L^2\frac{1-e^{-2\kappa t}}{2\kappa}\right)\exp\left(-\frac{\delta^2}{8}\frac{2\kappa}{1-e^{-2\kappa t}}\right)\right]P_te^{-W(x)}.
\end{align*}
We choose $\delta$ so that the term $e^{L^2}$ vanishes. In particular, we take
\[
\delta:=2\sqrt{2}L\frac{1-e^{-2\kappa t}}{2\kappa}
\]
to get
\begin{align*}
\frac{1}{t\sqrt{2}}|\EE\left[\nabla_v [f(X_t^x)] M_t^{x,u}\right]|&\le  \frac{Le^{-\kappa t}}{t\sqrt{2}}\left[2\sqrt{2}L\frac{1-e^{-2\kappa t}}{2\kappa} + 5\cdot 2^{1/4}\sqrt{\frac{1-e^{-2\kappa t}}{2\kappa}}\right]P_te^{-W(x)}\\
&\le \frac{Le^{-\kappa t}}{t\sqrt{2}}\left[2\sqrt{2}L\frac{2\kappa t}{2\kappa} + 5\cdot 2^{1/4}\sqrt{\frac{2\kappa t}{2\kappa}}\right]P_te^{-W(x)}\\
&=\left[2L^2e^{-\kappa t} + 5\cdot 2^{-1/2}\cdot 2^{1/4}\frac{Le^{-\kappa t}}{\sqrt{t}}\right]P_te^{-W(x)}\\
&\le 5e^{-\kappa t}\left[L^2+\frac{L}{\sqrt{t}}\right]P_te^{-W(x)}.
\end{align*}
\end{proof}
We now analyze the second term.
\begin{lemma}
\label{lem:2ndterm}
\[
\frac{1}{t}\int_0^t \left|\EE\left[\langle\nabla P_{t-s}e^{-W(X_s^x)},\nabla_{u,u}^2X_s^x\rangle\right]\right|ds\le  LKt\frac{e^{-\kappa t}}{2}P_te^{-W(x)}.
\]
\end{lemma}
\begin{proof}
By \cite[eq. (5.5.4)]{bakry2013analysis}, for any $y\in \R^n$, 
\begin{align}
\label{eq:dercommute}
|\nabla P_{t-s}e^{-W(y)}|\le e^{-\kappa (t-s)}P_{t-s}\left(|\nabla e^{-W(y)}|\right)=e^{-\kappa (t-s)}P_{t-s}\left(|\nabla W|e^{-W(y)}\right)\le L e^{-\kappa (t-s)} P_{t-s}e^{-W(y)}
\end{align}
so by \eqref{eq:dercommute} and  Lemma \ref{lem:2stderv},
\begin{align*}
&\frac{1}{t}\int_0^t \left|\EE\left[\langle\nabla P_{t-s}e^{-W(X_s^x)},\nabla_{u,u}^2X_s^x\rangle\right]\right|ds\le \frac{1}{t}\int_0^t \EE\left[|\nabla P_{t-s}e^{-W(X_s^x)}||\nabla_{u,u}^2X_s^x|\right]ds\\
&\le \frac{1}{t}  \int_0^t \EE\left[ e^{-\kappa (t-s)}L (P_{t-s}e^{-W(X_s^x)}) Ks e^{-\kappa s}\right]ds = LK\frac{e^{-\kappa t}}{t}P_te^{-W(x)}\int_0^t s   ds= LKt\frac{e^{-\kappa t}}{2}P_te^{-W(x)}.
\end{align*}
\end{proof}
We can now complete the proof of  Proposition \ref{prop:2nd_derivative_Euclidean}. By Lemma \ref{lem:martterm} and Lemma \ref{lem:2ndterm},
\begin{align*}
\nabla_{u,u}^2\log P_t e^{-W(x)}&=\frac{\nabla_{u,u}^2P_t e^{-W(x)}}{P_te^{-W(x)}}-|\nabla_u\log P_t e^{-W(x)}|^2\le \frac{\nabla_{u,u}^2P_t e^{-W(x)}}{P_te^{-W(x)}}\le 5e^{-\kappa t}\left[L^2+\frac{L}{\sqrt{t}}\right] +\frac{LK t}{2} e^{-\kappa t}\\
&=Le^{-\kappa t}\left[5L+\frac{5}{\sqrt{t}}+\frac{Kt}{2}\right].
\end{align*}

\section{The round sphere} \label{sec:sphere}
This section is devoted to the proof of Theorem \ref{thm_sphere}. Throughout this section, we let $\mu$ stand for the uniform probability measure on $\Sph$, the unit sphere in $\R^n$, and let $\nu = e^{-W(x)}d\mu$ be an $L$-log-Lipschitz perturbation of $\mu$. In other words, if $\gradsph$ is the spherical gradient, then for almost every $x \in \Sph$,
$$
|\gradsph W(x)| \leq L.
$$
As before, $X_t^x$ is the Langevin process associated with $\mu$, with semigroup $(P_t)$. Since $\mu$ is uniform, $X_t^x$ takes a particularly simple form $X_t^x = \omega_t$,
where $\omega_t$ is a standard Brownian motion on $\Sph$, with $\omega_0 = x$. 
We also write $X_t$ when $X_0 \sim \nu$.

As suggested by Proposition \ref{prop:2nd_derivative_to_Lipschitz}, the  proof of Theorem \ref{thm_sphere} will require bounding $\nabla_S^2\log P_t e^{-W(x)}$. The bound of $\nabla_S^2\log P_t e^{-W(x)}$ will be obtained, as in the Euclidean setting, by using a stochastic representation via Bismut's formula. The main result of this section is the following result.
\begin{proposition}
\label{prop:2nd_derivative_sphere}
For every $t\ge 0$ and $x \in \Sph$,
\begin{align*}
 \nabla_S^2\log P_t e^{-W(x)}\preceq  12\left(L + \frac{L^2}{\sqrt{n-2}}\right)e^{-(n-2)t}\left(\frac{1}{\sqrt{t}} + 1\right)g.
\end{align*}
\end{proposition}
 Before delving into the proof of Proposition \ref{prop:2nd_derivative_sphere}, let us show how Theorem \ref{thm_sphere} follows.
\begin{proof}[Proof of Theorem \ref{thm_sphere}]
Using 
\begin{align*}
\int_0^{\infty}e^{-(n-2) t} dt=\frac{1}{n-2}, \quad\quad\int_0^{\infty} \frac{e^{-(n-2) t}}{\sqrt{t}}dt=\sqrt{\frac{\pi}{n-2}},
\end{align*}
the combination of Proposition \ref{prop:2nd_derivative_to_Lipschitz} and Proposition \ref{prop:2nd_derivative_sphere} yields 
that $T$, the Langevin transport map, is Lipschitz with constant $\exp\left(12\left(\frac{L}{\sqrt{n-2}} + \frac{L^2}{n-2}\right)\left(\frac{1}{\sqrt{n-2}} + \sqrt{\pi}\right)\right) \leq \exp\left(35\left(\frac{L}{\sqrt{n-2}} + \frac{L^2}{n-2}\right)\right).$
\end{proof}

\subsection{Stochastic representation of the Hessian}
The proof of Proposition \ref{prop:2nd_derivative_sphere} will rely on a stochastic representation of the Hessian of the heat semigroup on the sphere, analogous to Lemma \ref{lem:Bismut}. Before introducing the particular form of this representation, we recall a few facts about the curvature of the sphere. Let $g$ be the metric on $\Sph$ induced from $\R^n$. Given vector fields $X,Y,Z$, the Riemannian curvature tensor on the sphere is given by
\begin{equation} \label{eq:curv_sphere}
    \Riem_S(X,Y)Z=g(Y,Z) X-g( X,Z) Y,
\end{equation}
and its trace, the Ricci curvature, is given by
\[
\Ric_S(Y,Z)=(n-2)g( Y,Z).
\]

For $t \geq 0$ and $x \in \Sph$, we use $\para_t: T_x\Sph \to T_{X_t^x}\Sph$ for the (random) \emph{parallel transport} operator between the tangent spaces at $x$ and at $X_t^x$. The role of the Jacobi flow will be played by the operators $Q_t: T_x\Sph \to T_{X_t^x}\Sph$ defined as,
\begin{equation} \label{eq:Qdef}
    Q_t:= e^{-(n-2)t}\para_t.
\end{equation}
An elementary calculation using the above formula for $\Ric_S$ shows that 
\[
\partial_tQ_t=-\Ric_S^{\sharp} Q_t; \hspace{3mm} Q_0 = \operatorname{id},
\]
which is analogous to \eqref{eq:Jacobi_Euclidean} for the Jacobi flow with a curvature term. Here, $\Ric_S^{\sharp}$ is defined by the canonical isomorphism $g(\Ric_S^{\sharp}(u), v)_x = \Ric_S(u,v)$ for all $ x \in M$, and $u, v \in T_xM.$

\begin{lemma}\textnormal{(\cite[Lemma 2.6]{cheng2021}\footnote{We rescale the generator of the heat equation so our $t$ corresponds to $2t$ in \cite{cheng2021}.})} \label{lem:bismut_sphere} \label{lem:bismutsphere}
    Let $f:\Sph \to \RR$ and $t > 0$. For any continuously differentiable $k:[0,t]\to \RR$  such that $k(0) = 1$ and $k(t) = 0$, we have, for any $x \in \Sph$ and $v,u\in T_x\Sph$,
    \begin{align*}
        \label{hess_sphere} \Hesssph P_tf (v, u) =
           &\ \ - \sqrt{2}\EE\left[g(\gradsph f(X^x_t),Q_t(v)) \int_0^t g(Q_s(\dot{k}(s)u), \para_sdB_s)\right]\\
           &+\sqrt{2}\EE\left[g\left( \gradsph  f(X^x_t), Q_t\int_0^tQ_s^{-1}\Riem_S(\para_sdB_s,Q_s(v))Q_s(k(s)u)\right) \right],
    \end{align*}
    where $B_s$ is a standard Brownian motion on $T_x\Sph$.
\end{lemma}
Lemma \ref{lem:bismut_sphere} allows for some flexibility with the choice of the function $k$. For our use, we fix the following specific choice 
\begin{equation} \label{eq:kchoice}
    k(s) := 1-\frac{s}{t}\quad \forall s\in [0,t].
\end{equation}

\subsection{Proof of Proposition \ref{prop:2nd_derivative_sphere}}
As in the Euclidean setting, we need to bound the first and second terms of Lemma \ref{lem:bismut_sphere} separately. The first term is analogous to the first term of Lemma \ref{lem:Bismut} where we need to bound the martingale $\int_0^t \langle Q_s(\dot{k}(s)w), \para_sdB_s\rangle$. The second term of Lemma \ref{lem:Bismut} is absent in the sphere setting since $\mu$ is the uniform measure. However, due to the curvature of the sphere, we get the second term in Lemma \ref{lem:bismut_sphere}. As we will show, this term \emph{also} involves a martingale. More precisely, the term $$Q_t\int_0^tQ_s^{-1}\Riem_S(\para_sdB_s,Q_s(v))Q_s(k(s)v)$$is a martingale, which is a feature of the constant curvature of the sphere, absent in the general manifold setting.
In light of this, to bound $\Hesssph P_te^{-W}$, we shall require the following martingale argument, which extends the bounds obtained in the Euclidean setting.
\begin{lemma} \label{lem:martingale_bound_sphere}
Let $f:\Sph \to \RR$ be an $L$-log-Lipschitz function and let $(M_t)$ be a martingale on $T_x\Sph$, for some $x \in \Sph$. Assume that almost surely, $M_t$ has bounded quadratic variation, that is
$$\langle M \rangle_t \leq \varphi(t) \hspace{2mm} a.s.,$$
for some $\varphi(t) > 0$.
Then
$$\EE[|g( \gradsph  f(X^x_t),\para_t M_t )|] \leq \left(6L + 2\frac{L^2}{\sqrt{n-2}}\right)\sqrt{\varphi(t)}\EE[f(X^x_t)].$$
\end{lemma}
\begin{proof}
Since $\mu$ is the uniform measure on $\Sph$, Lemma \ref{lem:reverseHollder} and Remark \ref{rmk:cdk} yield
$$\EE[|\gradsph f(X^x_t)|^2] \leq L^2\EE[f(X^x_t)^2] \leq L^2\exp\left(L^2\frac{1 - e^{-2(n-2) t}}{n-2}\right)\EE[f(X^x_t)]^2,$$
where the first inequality is the $L$-log-Lipschitz condition coupled with the chain rule.
As in the Euclidean setting, we have, for any $\delta > 0$,
\begin{equation*}
\EE[|\gradsph f(X^x_t)|  |M_t|] \leq \delta \EE[|\gradsph f(X^x_t)|] + \EE[|\gradsph f(X^x_t)|^2]^{1/2}\EE[|M_t|^4]^{1/4}\PP[|M_t| \geq \delta]^{1/4},
\end{equation*}
as well as the bound $\EE[|M_t|^4]^{1/4} \leq 5\varphi(t)^{1/2}$ which follows from the Burkholder-Davis-Gundy inequality \cite[Proposition 3.26]{KaratzasShreve} and the bound on the quadratic variation. Moreover, using Lemma \ref{lem_deviation-mart}, $\PP[|M_t| \geq \delta] \leq 2\exp\left(-\frac{\delta^2}{2\varphi(t)}\right)$. It follows that
\begin{align*}
\EE[|\gradsph f(X^x_t)|  |M_t|] &\leq \delta\EE[|\gradsph f(X^x_t)|]+ \left(5\cdot 2^{1/4} L\exp\left(L^2\frac{1 - e^{-2(n-2) t}}{2(n-2)}\right)\varphi(t)^{1/2}\left(-\frac{\delta^2}{8\varphi(t)}\right)\right)\EE[f(X_t)]\\
&\le  \left(\delta L+5\cdot 2^{1/4} L\exp\left(L^2\frac{1 - e^{-2(n-2) t}}{2(n-2)}\right)\varphi(t)^{1/2}\exp\left(-\frac{\delta^2}{8\varphi(t)}\right)\right)\EE[f(X_t)],
\end{align*}
where the second inequality used $\EE[|\gradsph f(X^x_t)|] \leq L\EE[f(X^x_t)]$ since $f$ is $L$-log-Lipschitz. We choose $\delta$ so that the term $e^{L^2}$ vanishes. In particular, we take $\delta = 2\sqrt{2}L\sqrt{\varphi(t)}\sqrt{\frac{1 - e^{-2(n-2) t}}{2(n-2)}}$, we get
\begin{align*}\EE[|\gradsph f(X^x_t)|  |M_t|] &\leq \sqrt{\varphi(t)}\left(2\sqrt{2}L^2\sqrt{\frac{1 - e^{-2(n-2)t}}{2(n-2)}} + 5\cdot 2^{1/4}L \right)\EE[f(X^x_t)] \\
&\leq \left(6L + 2\frac{L^2}{\sqrt{n-2}}\right)\sqrt{\varphi(t)}\EE[f(X^x_t)].
\end{align*}
\end{proof}
We now bound the two terms appearing in Lemma \ref{lem:bismut_sphere} using Lemma \ref{lem:martingale_bound_sphere}. For brevity, let us use $f:=e^{-W}$. We start by analyzing the first term.  
\begin{lemma}
\label{lem:second_term_sphere}  
Fix $x \in \Sph$. For every unit vectors $v,u\in T_x\Sph$ and $t>0$,
\begin{align*}
      \left|\EE\left[g(\gradsph f(X^x_t),Q_t(v)) \int_0^t g( Q_s(\dot{k}(s)u), \para_sdB_s)\right]\right| \leq 6\frac{e^{-(n-2)t}}{\sqrt{t}}\left(L + \frac{L^2}{\sqrt{n-2}}\right)P_tf(x).
  \end{align*}
\end{lemma}
\begin{proof}
Using \eqref{eq:Qdef} we see that the term $\EE\left[g(\gradsph f(X^x_t),Q_t(v)) \int_0^t g(Q_s(\dot{k}(s)u), \para_sdB_s)\right]$ can be written as $e^{-(n-2)t}\EE[g(\gradsph f(X^x_t),\para_t M_t) ]$ where 
\[
M_t=\left( \int_0^t g(Q_s(\dot{k}(s)u), \para_sdB_s)\right)v
\]
is a martingale. The  quadratic variation of $M_t$ is equal to, using \eqref{eq:kchoice},
$$\int_0^te^{-2(n-2)s}\dot{k}(s)^2ds= \frac{1}{t^2}\int_0^te^{-2(n-2)s}=\frac{1}{t^2}\frac{1-e^{-2(n-2)t}}{2(n-2)}\le\frac{1}{t}=:\varphi(t)
$$
where the inequality follows from $1-e^{-2(n-2)t}\le 2(n-2)t$.
The proof is complete by Lemma \ref{lem:martingale_bound_sphere}.
\end{proof}

We now analyze the second term where we take $v=u$.
\begin{lemma}
\label{lem:first_term_sphere}  
Fix $x \in \Sph$. For every unit vector $v\in T_x\Sph$ and $t>0$,
    \begin{align*} 
        \left|\EE\left[g\left( \gradsph  f(X^x_t), Q_t\int_0^tQ_s^{-1}\Riem_S(\para_sdB_s,Q_s(v))Q_s(k(s)v)\right) \right]\right| \leq 6e^{-(n-2)t}\left(L + \frac{L^2}{\sqrt{n-2}}\right)P_tf(x).
    \end{align*}
\end{lemma}    
\begin{proof} 
We use the expressions \eqref{eq:curv_sphere} and \eqref{eq:Qdef} to write
\begin{align*}
&Q_s^{-1}\Riem_S(\para_sdB_s,Q_s(v))Q_s(k(s)v)=\left(e^{-(n-2)s}\right)^{-1}\para_s^{-1}\Riem_S(\para_sdB_s,e^{-(n-2)s}\para_s v)e^{-(n-2)s}k(s)\para_s v\\
&=e^{-(n-2)s}k(s)\para_s^{-1}\Riem_S(\para_sdB_s,\para_s v) \para_sv=e^{-(n-2)s}k(s)\para_s^{-1}\left[g\left( \para_sv,\para_sv\right) \para_s  dB_s-g\left( \para_s dB_s,\para_s v,\right) \para_s v\right]\\
&=e^{-(n-2)s}k(s)\para_s^{-1}\left[\para_s dB_s-g\left( \para_s dB_s,\para_s v\right)\para_sv\right]=e^{-(n-2)s}k(s)\left[ dB_s-g\left( dB_s,v\right) v\right],
\end{align*}
where in the last equation we used that $\para_s$ is a linear isometry. The term $dB_s-g\left( dB_s, v\right) v$ is a Brownian motion in the hyperplane orthogonal to $v$ inside $T_x\Sph$, so it follows that $\int_0^tQ_s^{-1}\Riem_S(//_sdB_s,Q_s(v))Q_s(k(s)v)$ is a martingale in $T_x\Sph$ whose quadratic variation is bounded from above by (using $0\le k(s)\le 1$ for all $s\in [0,t]$),
$$(n-2)\int_0^t{e^{-2(n-2)s}k(s)^2dr}\leq (n-2)\int_0^t{e^{-2(n-2)s}dr}\leq \frac{n-2}{2(n-2)} =\frac{1}{2},$$
Therefore, we can view the term 
\[
\EE\left[g\left( \gradsph  f(X^x_t), Q_t\int_0^tQ_s^{-1}\Riem_S(\para_sdB_s,Q_s(v))Q_s(k(s)v)\right)\right]
\]
as of the form 
$$e^{-(n-2)t}\EE[g(\gradsph  f(X^x_t),\para_t M_t )],$$ where $M_t$ is a martingale on $T_x\Sph$ with quadratic variation bounded by $\frac{1}{2}$. The proof is complete by Lemma \ref{lem:martingale_bound_sphere}.
\end{proof}

We can now complete the proof of Proposition \ref{prop:2nd_derivative_sphere}. By Lemma \ref{lem:second_term_sphere} and Lemma \ref{lem:first_term_sphere},
\begin{align*}
\nabla_S^2\log P_t e^{-W(x)}(u,u)&=\frac{\nabla_S^2P_t e^{-W(x)}(u,u)}{P_te^{-W(x)}}-g(\nabla_S\log P_t e^{-W(x)},u)^2\le \frac{|\nabla_S^2P_t e^{-W(x)}(u,u)|} {P_te^{-W(x)}} \\
&\leq 12\left(L + \frac{L^2}{\sqrt{n-2}}\right)e^{-(n-2)t}\left(\frac{1}{\sqrt{t}} + 1\right).
\end{align*}

\section{General manifolds} \label{sec:genmanifold}
Let $(M,g,\mu)$ with $\mu = e^{-V}d\operatorname{Vol}$ be a weighted Riemannian manifold, with tangent bundle $TM$. We start by introducing the relevant concepts from Riemannian geometry.
\paragraph{Notions from Riemannian geometry:}
\begin{itemize}
    \item If $S$ is a symmetric tensor on $TM$, we denote $S^{\#} : TM \longrightarrow TM$ the operator such that $g(S^{\#}(u), v)_x = S(u,v), ~~ \forall x \in M$, $u, v \in T_xM$ where we suppress the dependence of $S$ on $x$. 

    \item $\Riem$ denotes the Riemannian curvature $4$-tensor on $(M,g)$. Its norm is defined as
    $$||\Riem||_\infty:=\sup \left\{\sum_{\rev{i=1}}^d g( \Riem(x)(e_i, u)v, S^{\#}(e_i)), x \in M, |u|, |v| \leq 1, |S|_{\textnormal{op}} \leq 1\right\}$$
    where $u, v \in T_xM$, $\{e_i\}$ is any basis of $T_xM$,  and $S$ is a symmetric 2-tensor.
    
    \item $\Ric$ denotes the Ricci curvature tensor on $(M,g)$. 
     \item $d^*\Riem$ is defined as
    $$g(d^*\Riem(u,v), w) = g((\nabla_w \Ric^{\#})(u), v) - g((\nabla_v \Ric^{\#})(w), u).$$  
    \item $\Ric_V$ denotes the weighted Ricci curvature tensor of Bakry-\'Emery on $(M, g, \mu)$, that is,
    $$\Ric_V := \Ric + \nabla_M^2 V$$ 
    where $\nabla_M$ and $\nabla_M^2$ are, respectively, the gradient and Hessian on the manifold $M$. 
\end{itemize}

The manifold $(M,g,\mu)$ satisfies the $\mathrm{CD}(\kappa,\infty)$ \emph{curvature-dimension}, for $\kappa \in \RR$, if
$$\Ric_V \succeq \kappa g.$$
With these definitions, we state our main result, which is the precise form of Theorem \ref{thm:manifold_intro}. 
\begin{theorem}
\label{thm:manifold}
Let $(M,\rev{g},\mu)$ be a weighted Riemannian manifold with $\mu = e^{-V}d\mathrm{Vol},$ and let $\nu = e^{-W}d\mu$ be another measure on $M$, with $W$ an $L$-Lipschitz function. Assume that $||\Riem ||_\infty < \infty$, $\Ric_V \succeq \kappa g$ with $\kappa > 0$, and that 
$$\beta := ||\nabla \Ric_V^{\#} + d^*\Riem + \Riem(\nabla V)||_\infty <\infty.$$ Then,
the Langevin transport map from $\mu$ onto $\nu$ is Lipschitz with constant $\exp\left(L\left(\frac{e^{\frac{L^2}{2\kappa}}}{\sqrt{\kappa}} + e^{\frac{L^2}{2\kappa}}\frac{\|\Riem\|_\infty}{\kappa^{\frac{3}{2}}} + \frac{\beta}{\kappa^2}\right)\right)$. 
\end{theorem}
On the assumptions, the lower bound on the Ricci curvature tensor is the natural analogue of the uniform convexity assumption in the Euclidean setting, while a bound on $\nabla \Ric^{\#}_V$ is the natural counterpart to the third derivative bound on $V$ in the Euclidean setting. The other terms we have to control are purely geometric and  vanish if we consider a Euclidean space or the uniform measure on the sphere.

As in the previous section, we shall require a bound on the Hessian along the Langevin semigroup. The following result is the analog of Proposition \ref{prop:2nd_derivative_Euclidean} and Proposition \ref{prop:2nd_derivative_sphere} but for general manifolds. The estimate of Proposition \ref{prop:2nd_derivative_Manifold} is not strong enough to yield the sharp results of Theorem \ref{thm:Euclidean} and Theorem \ref{thm_sphere} which necessitate Proposition \ref{prop:2nd_derivative_Euclidean} and Proposition \ref{prop:2nd_derivative_sphere}. It is this estimate which leads to the double-exponential estimate on the Lipschitz constant of the transport maps.

\begin{proposition}
\label{prop:2nd_derivative_Manifold}
For every $t\ge 0$ and $x \in M$,
\begin{align*}
\nabla_M^2\log P_t e^{-W(x)}\preceq e^{-\kappa t}L\left(\left(\frac{\sqrt{\kappa}}{\rev{\sqrt{e^{2\kappa t}-1}} }+ \frac{\|\Riem\|_\infty }{\sqrt{\kappa}}\right)e^{L^2\left(\frac{1-e^{-2\kappa t}}{2\kappa}\right)} + \frac{\beta }{\kappa}\right)g.
\end{align*}
\end{proposition}

\begin{proof}
By Lemma \ref{lem:reverseHollder} and Remark \ref{rmk:cdk}, if $f$ is positive and $L$-log-Lipschitz then 
\begin{align}
\label{eq:reverse_Holder}
P_t(|\nabla_M f|^2)\le \exp\left(L^2\frac{1-e^{-2\kappa t}}{\kappa}\right)(P_t|\nabla_M f|)^2\le  L^2\exp\left(L^2\frac{1-e^{-2\kappa t}}{\kappa}\right)(P_t f)^2.
\end{align}
Hence, by \cite[Theorem 2.5]{cheng2021}\footnote{We rescale the generator of the heat equation so our $t$ corresponds to $2t$ in \cite{cheng2021}.} and \eqref{eq:reverse_Holder}, 
\begin{align}
\label{eq:CTW}
\begin{split}
\nabla_M^2 P_t e^{-W(x)}&\preceq e^{-\kappa t}\left(\left(\frac{\sqrt{\kappa}}{\sqrt{e^{2\kappa t}-1}} + \frac{\|\Riem\|_\infty }{\sqrt{\kappa}}\right)(P_t|\nabla_M e^{-W}|^2)^{1/2} + \frac{\beta }{\kappa} P_t|\nabla_M e^{-W}|\right)g\\
&\preceq e^{-\kappa t}\left(\left(\frac{\sqrt{\kappa}}{\sqrt{e^{2\kappa t}-1}} + \frac{\|\Riem\|_\infty }{\sqrt{\kappa}}\right)Le^{L^2\left(\frac{1-e^{-2\kappa t}}{2\kappa}\right)}P_t e^{-W} + \frac{\beta L}{\kappa}P_te^{-W}\right)g\\
\end{split}
\end{align}
so \eqref{eq:CTW} yields
\begin{align*}
\nabla_M^2 \log P_t e^{-W(x)}&=\frac{\nabla_M^2 P_t e^{-W(x)}}{P_t e^{-W(x)}}-\left(\nabla_M \log P_t e^{-W(x)}\right)^{\otimes 2}\preceq \frac{\nabla_M^2 P_t e^{-W(x)}}{P_t e^{-W(x)}}\\
&\preceq e^{-\kappa t}L\left(\left(\frac{\sqrt{\kappa}}{\sqrt{e^{2\kappa t}-1}} + \frac{\|\Riem\|_\infty }{\sqrt{\kappa}}\right)e^{L^2\left(\frac{1-e^{-2\kappa t}}{2\kappa}\right)} + \frac{\beta }{\kappa}\right)g.
\end{align*}
\end{proof}
\begin{proof}[Proof of Theorem \ref{thm:manifold}]
Proposition \ref{prop:2nd_derivative_Manifold} implies, for every $x \in M$ and unit vector $u\in T_xM$,
\begin{align*}
    \int\limits_0^\infty \left|\nabla_M^2 \log P_te^{-W(x)}(u,u)\right|dt &\leq \int\limits_0^\infty e^{-\kappa t}L\left(\left(\frac{\sqrt{\kappa}}{\sqrt{e^{2\kappa t}-1}} + \frac{\|\Riem\|_\infty }{\sqrt{\kappa}}\right)e^{\frac{L^2}{2\kappa}} + \frac{\beta }{\kappa}\right)dt\\
    &= L\left(\sqrt{\kappa}e^{\frac{L^2}{2\kappa}}\int\limits_0^\infty \frac{e^{-\kappa t}}{\sqrt{e^{2\kappa t}-1}}dt + \left(e^{\frac{L^2}{2\kappa}}\frac{\|\Riem\|_\infty}{\sqrt{\kappa}} + \frac{\beta}{\kappa}\right)\int\limits_0^\infty e^{-\kappa t} dt\right)\\
    &= L\left(\frac{e^{\frac{L^2}{2\kappa}}}{\sqrt{\kappa}} + e^{\frac{L^2}{2\kappa}}\frac{\|\Riem\|_\infty}{\kappa^{\frac{3}{2}}} + \frac{\beta}{\kappa^2}\right).
\end{align*}
The proof is complete by Proposition \ref{prop:2nd_derivative_to_Lipschitz}.
\end{proof}

\section{The inverse of the Langevin transport map}
\label{sec:rev_lip}
In this section, we focus on the inverse of the Langevin transport map. Recall, from \eqref{eq:Langevin_flow}, that $(S_t)_{t\geq0}$ stands for forward maps along the Langevin flow. Thus we have that $S:=\lim\limits_{t\to \infty} S_t = T^{-1}$, when $T$ is the Langevin transport map. We shall prove the following precise form of Theorem \ref{thm:meta_reverse_intro}.
\begin{theorem}
\label{thm:meta_reverse}
$~$
\begin{enumerate}
    \item Let $\mu$ and $\nu$ be as in Theorem \ref{thm:Euclidean}. Then $S$ is Lipschitz with constant $\exp\left(\frac{21}{2}\frac{L^2}{\kappa}+\frac{5\sqrt{\pi}L}{\sqrt{\kappa}}+\frac{LK}{2\kappa^2}\right)$
    \item Let $\mu$ and $\nu$ be as in Theorem \ref{thm_sphere}. Then $S$ is Lipschitz with constant \rev{$
 \exp\left(35 L+ \frac{71}{2}L^2\right)$}.
 \item Let $\mu$ and $\nu$ be as in Theorem \ref{thm:manifold}. Then $S$ is Lipschitz with constant $
 \exp\left(L\left(\frac{e^{\frac{L^2}{2\kappa}}}{\sqrt{\kappa}} + e^{\frac{L^2}{2\kappa}}\frac{\|\Riem\|_\infty}{\kappa^{\frac{3}{2}}} + \frac{\beta}{\kappa^2}\right)+\frac{L^2}{2\kappa}\right)$.
\end{enumerate}
\end{theorem}

The proof of Theorem \ref{thm:meta_reverse} is analogous to the proofs of the Lipschitz properties of the Langevin transport map $T$ from $\mu$ to $\nu$. With the same argument as in the proof of Proposition \ref{prop:2nd_derivative_to_Lipschitz}, without reversing the map, we get:
\begin{proposition} \label{prop_reverse_lip}
If for all $t\ge 0$
\[
\nabla_M^2\log P_te^{-W}\succeq  -\ell_t g,
\]
then $S$ is $\exp\left(\int_0^{\infty}\ell_tdt\right)$-Lipschitz.
\end{proposition}
The next result shows how to combine log-Lipschitz properties together with Hessian estimates to get a lower bound on the Hessian as in Proposition \ref{prop_reverse_lip}.
\begin{proposition}
\label{prop:meta_reverse}
Let $(M,g,\mu)$ be weighted Riemannian manifold with the associated Langevin semigroup $(P_t)$ and let $|\cdot|:=g(\cdot,\cdot)$. Suppose that we have:
\begin{enumerate}
    \item There exists $\kappa>0$ such that for every test function $f$,
    $$|\nabla_M P_tf| \leq e^{-\kappa t}P_t|\nabla_M f|.$$
    \item There exists $\alpha:[0,\infty)\times[0,\infty) \to \R$ such that for every $L$-log-Lipschitz  nonnegative function $f$, $$\frac{\nabla_M^2P_tf}{P_tf}\succeq -\alpha(t,L)g.$$
\end{enumerate}
Then, for every $L$-log-Lipschitz  function $f$,
$$\nabla_M^2\log P_tf \succeq  -(\alpha(t,L)+L^2e^{-2\kappa t})g.$$
\end{proposition}

\begin{proof}
Item 1 implies that
\begin{align*}
| \nabla_M \log P_t f| &= \frac{|\nabla_M P_tf|}{|P_tf|}\leq  e^{-\kappa t}\frac{P_t|\nabla_M f|}{|P_t f|}= e^{-\kappa t}\frac{P_t|f\nabla_M \log f|}{|P_t f|}\le Le^{-\kappa t}
\end{align*}
so, by item 2, for every tangent vector $u$,
\[
\nabla_M^2 \log P_t f(u,u)=\frac{\nabla_M^2 P_t f(u,u)}{P_t f}-g(\nabla_M \log P_t f,u)^2\ge \rev{- (\alpha(t,L)+L^2e^{-2\kappa t})}.
\]
\end{proof}

We can now complete the proof of Theorem \ref{thm:meta_reverse}. 

\begin{enumerate}
    \item The assumption $\nabla^2 V(x)	\succeq \kappa \mathrm{I}$ together with the classical Bakry-\'Emery gradient estimate \cite[eq. (5.5.4)]{bakry2013analysis} shows that item 1 of Proposition \ref{prop:meta_reverse} holds with $\kappa$. As for item 2, it follows from the proof of Proposition \ref{prop:2nd_derivative_Euclidean} that it holds with
    \[
    \alpha(t,L)=Le^{-\kappa t}\left[5L+\frac{5}{\sqrt{t}}+\frac{Kt}{2}\right].
    \]
 Hence, the condition in Proposition \ref{prop_reverse_lip} is satisfied with 
 \[
 \ell_t =Le^{-\kappa t}\left[5L+\frac{5}{\sqrt{t}}+\frac{Kt}{2}\right]+L^2e^{-2\kappa t},
 \]
which completes the proof by integration from $t=0$ to $t=\infty$ (see the proof of Theorem \ref{thm:Euclidean}).  
\item   The sphere $\Sph$ satisfies the curvature-dimension condition $\mathrm{CD}(n-2,\infty)$ so the Bakry-\'Emery gradient estimate \cite[eq. (5.5.4)]{bakry2013analysis} can be applied to show that item 1 of Proposition \ref{prop:meta_reverse} holds with $n-2$. As for item 2, it follows from the proof of Proposition \ref{prop:2nd_derivative_sphere} that it holds with
    \[
    \alpha(t,L)= \rev{35}\left(L + \frac{L^2}{\sqrt{n-2}}\right)e^{-(n-2)t}\left(\frac{1}{\sqrt{t}} + 1\right).
\]
 Hence, the condition in Proposition \ref{prop_reverse_lip} is satisfied with 
 \[
 \ell_t =35\left(L + \frac{L^2}{\sqrt{n-2}}\right)e^{-(n-2)t}\left(\frac{1}{\sqrt{t}} + 1\right)+L^2e^{-2(n-2) t},
 \]
 which completes the proof by integration from $t=0$ to $t=\infty$ (see the proof of Theorem \ref{thm_sphere}).

 \item By assumption, $M$ satisfies the curvature-dimension condition $\mathrm{CD}(\kappa,\infty)$ so the Bakry-\'Emery gradient estimate \cite[eq. (5.5.4)]{bakry2013analysis} can be applied to show that item 1 of Proposition \ref{prop:meta_reverse} holds with $\kappa$. As for item 2, it follows from the proof of Proposition \ref{prop:2nd_derivative_Manifold} that it holds with
    \[
    \alpha(t,L)=
 e^{-\kappa t}L\left(\left(\frac{\sqrt{\kappa}}{\sqrt{e^{2\kappa t}-1}} + \frac{\|\Riem\|_\infty }{\sqrt{\kappa}}\right)e^{L^2\left(\frac{1-e^{-2\kappa t}}{2\kappa}\right)} + \frac{\beta }{\kappa}\right).
\]
Hence, the condition in Proposition \ref{prop_reverse_lip} is satisfied with 
 \[
 \ell_t =e^{-\kappa t}L\left(\left(\frac{\sqrt{\kappa}}{\sqrt{e^{2\kappa t}-1}} + \frac{\|\Riem\|_\infty }{\sqrt{\kappa}}\right)e^{L^2\left(\frac{1-e^{-2\kappa t}}{2\kappa}\right)} + \frac{\beta }{\kappa}\right)+L^2e^{-2\kappa t},
 \]
 which completes the proof by integration from $t=0$ to $t=\infty$ (see the proof of Theorem \ref{thm:manifold}).

\end{enumerate}

\section{Applications} 
\label{sect_applications}

As is classical, Lipschitz transport maps can be used to prove functional inequalities, by transferring them from the source measure to the target measure. In this section, we discuss some results on functional inequalities for Lipschitz perturbations that appear to be new. 

\subsection{Dimension-free Gaussian isoperimetric inequalities}

One of the strongest functional inequalities for the Gaussian measure is the Gaussian isoperimetric inequality, which states that among all sets with given Gaussian mass, the minimal Gaussian perimeter is achieved for half-spaces. The general functional inequality is defined as follows: 

\begin{definition}
Let $\mu$ be a probability measure on a metric space. Given a Borel set $A$, we define its $t$-enlargement as $\{x; d(x,A) \leq t\}$ and its boundary measure (or Minkowski content) as 
$$\mu^+(A) := \liminf_{t \rightarrow 0} t^{-1}\mu(A^t \backslash A).$$

The probability measure $\mu$ is said to satisfy a Gaussian isoperimetric inequality with constant $\alpha > 0$ if for any Borel set we have
$$\mu^+(A) \geq \alpha I(\mu(A)),\quad I = \varphi \circ \phi^{-1}$$
where $\varphi(x) = (2\pi)^{-1/2}\exp(-x^2/2)$ and $\phi(x) = \int_{-\infty}^x{ \varphi(t)dt}$.  
\end{definition}

The function $I$ corresponds to the perimeter of a half-space with the right Gaussian mass, and the Gaussian measure satisfies this inequality with constant $1$ in all dimensions. This inequality was generalized to uniformly log-concave measures in \cite{BaLe96}. The functional form of the Gaussian isoperimetric inequality is Bobkov's inequality \cite{Bob97}. It is immediate that Gaussian isoperimetric inequalities can be transferred by Lipschitz transport maps (see \cite{cordero2002some}), so we get the following as a corollary of Theorem \ref{thm:Euclidean}: 
\begin{theorem}
If $\mu$ is a log-Lipschitz perturbation of a standard Gaussian measure on $\mathbb{R}^d$, then it satisfies a Gaussian isoperimetric inequality with a dimension-free constant. 
\end{theorem}

The same result holds in the general context of Theorems \ref{thm:Euclidean}, \ref{thm_sphere} and \ref{thm:manifold}. For the sphere, we can transport the sharp isoperimetric inequality on the sphere, which contains a dimensional improvement compared with the Gaussian isoperimetric inequality. 

To our knowledge, this is the first result on dimension-free isoperimetric inequalities for log-Lipschitz perturbations. A result of Miclo (see \cite[Lemma 2.1]{bardet_conv}) states that Lipschitz perturbations of uniformly log-concave measures can be recast as bounded perturbations, allowing to use the Holley-Stroock lemma to deduce functional inequalities. However, the constants obtained in that way are dimension-dependent. Some of the results of \cite[Section 5]{BaKo08} can be used to prove functional inequalities for Lipschitz perturbations of uniformly-convex measures, but the constants depend on exponential moments of the unperturbed measure, and are hence dimensional.

Gaussian isoperimetric inequalities imply other functional inequalities, including in particular logarithmic Sobolev inequalities: 

\begin{definition}
  A probability measure $\mu$ on a manifold is said to satisfy a logarithmic Sobolev inequality if for any locally Lipschitz function $f$ such that $|\nabla f| \in L^2(\mu)$, we have
$$\int{f^2\log f^2 d\mu} - \left(\int{f^2d\mu}\right)\log\left(\int{f^2d\mu}\right) \leq 2C_{LSI}\int{|\nabla f|^2d\mu}.$$  
\end{definition}

The standard Gaussian satisfies a logarithmic Sobolev inequality, with sharp constant $C_{LSI} = 1$. This inequality, which is the sharpest possible analogue of Sobolev inequalities for the Gauss space, implies for example dimension-free concentration bounds, and has found many applications in statistics and statistical physics. 
We refer to \cite{bakry2013analysis} for an overview. 

For logarithmic Sobolev inequalities, Aida and Shigekawa \cite{aida_lsi} showed that such Lipschitz perturbations satisfy them, with a dimension-free constant that has not been made explicit. Gaussian isoperimetric inequalities are strictly stronger than logarithmic Sobolev inequalities. 

As a corollary, we recover functional inequalities for Gaussian mixtures with compactly supported mixing measures, which can be rewritten as Lipschitz perturbations of a Gaussian measure (see the discussion at the end of \cite[Section 2]{bardet_conv}). 

There are some other estimates for which it is difficult to provide direct proofs, but that can easily be extended from one measure to another via Lipschitz transport maps, including general eigenvalue estimates for diffusion generators \cite{milman2018spectral} and sharpened integrability bounds for non-Lipschitz functions
\cite{IvRu20, CMP20}. 

\subsection{A growth estimate for optimal transport maps}

Our argument does not apply to quadratic optimal transport maps. Nonetheless, it would be natural to expect similar Lipschitz estimates for them. As a hint towards such a result, we remark that the Gaussian concentration inequality implied by our Lipschitz estimate for non-optimal maps, combined with the sub-Gaussian case of \cite[Theorem 1.1]{colombo_bounds}, implies controlled linear growth for optimal transport maps, in the following way: 

\begin{proposition}
Let $\mu$ be a centered and $L$-log-Lipschitz perturbation of a standard Gaussian measure on $\mathbb{R}^d$. Then the quadratic optimal transport map $\nabla \varphi$  sending $\gamma$ onto $\mu$ satisfies
$$|\nabla \varphi(x)| \leq C_1 \exp(C_2 \max(L, L^2))\sqrt{d + |x|^2}$$
where $C_1$ and $C_2$ are universal constants. 
\end{proposition}

In view of our results, it is natural to expect that Lipschitz estimates hold for optimal transport maps. Even the Gaussian case is open at this time: 
\begin{conjecture} \label{conj:transport}
The Brenier map from the standard Gaussian measure onto a log-Lipschitz perturbation of it is globally Lipschitz, and its norm can be bounded independently of the dimension. 
\end{conjecture}

Note that such a statement would already significantly extend the Gaussian sub-case of \cite[Theorem 1.1]{colombo_lipschitz} when the perturbation is smooth. Similar questions can be raised for more general measures, as well as in the Riemannian setting, but the Gaussian case seems like a good starting point. Since the heat flow map and the Brenier transport map coincide in dimension one, this conjecture does hold on $\R$. 

Let us conclude with some comments on what is missing to prove this conjecture. The optimal transport map is the unique weak solution to the Monge-Amp\`ere equation
$$e^{-V} = e^{-V(\nabla \varphi) - W(\nabla \varphi)}\det\nabla^2 \varphi$$
among convex functions. This is a fully nonlinear second-order PDE. One can start by investigating the linearized equation. If we replace $W$ by a small perturbation $\varepsilon W - c_\varepsilon$, with $c_\varepsilon$ a constant to enforce unit mass for the target distribution, and look for a solution of the form $\nabla \varphi = x + \varepsilon \nabla h$, the linearized equation as $\varepsilon$ goes to zero is 
\begin{equation} \label{eq_pois_V}
\Delta h - \nabla V \cdot \nabla h = W - \int{We^{-V}dx}.
\end{equation}
The operator on the left-hand side is precisely the generator of the diffusion process \eqref{eq:SDElangevin}. This linearization highlights the connection between quadratic optimal transport and the heat flow map investigated here: the linearization of both constructions gives rise to the same PDE. 

To prove a regularity estimate for a nonlinear PDE, it is natural to start from a proof for the linearized equation. Regularity for solutions to Poisson equations such as \eqref{eq_pois_V} was investigated in \cite{GDVM19, FSX19a, FSX19b} using the stochastic representation of solutions, Malliavin calculus and Bismut's formula. These are precisely the kind of tools we used in this work. 

Since optimal transport maps do not admit a similar stochastic representation, it would be natural to start by looking for a PDE proof of the regularity estimates of \cite{GDVM19} on \eqref{eq_pois_V}. To our knowledge, this has not been successfully investigated and is at this point a barrier to proving Conjecture \ref{conj:transport}. The problem is of course a variant of many well-studied elliptic regularity problems. The main difficulties are that we work in a non-compact setting, allowing for unbounded solutions, and seek explicit \emph{dimension-free} estimates. This last point in particular has rarely been investigated with PDE methods. 

\vspace{1cm}

\textbf{Data availability statement}: We do not make use of or generate data sets.

\bibliographystyle{plain}
\bibliography{manifold_transport.bib}

\begin{thebibliography}{10}

\bibitem{aida_lsi}
Shigeki Aida and Ichiro Shigekawa.
\newblock Logarithmic {S}obolev inequalities and spectral gaps: perturbation
  theory.
\newblock {\em J. Funct. Anal.}, 126(2):448--475, 1994.

\bibitem{bakry2013analysis}
Dominique Bakry, Ivan Gentil, and Michel Ledoux.
\newblock {\em Analysis and geometry of {M}arkov diffusion operators}, volume
  348 of {\em Grundlehren der mathematischen Wissenschaften [Fundamental
  Principles of Mathematical Sciences]}.
\newblock Springer, Cham, 2014.

\bibitem{BaLe96}
Dominique Bakry and Michel Ledoux.
\newblock L\'{e}vy-{G}romov's isoperimetric inequality for an
  infinite-dimensional diffusion generator.
\newblock {\em Invent. Math.}, 123(2):259--281, 1996.

\bibitem{bardet_conv}
Jean-Baptiste Bardet, Natha\"{e}l Gozlan, Florent Malrieu, and Pierre-Andr\'{e}
  Zitt.
\newblock Functional inequalities for {G}aussian convolutions of compactly
  supported measures: explicit bounds and dimension dependence.
\newblock {\em Bernoulli}, 24(1):333--353, 2018.

\bibitem{BaKo08}
Franck Barthe and Alexander~V. Kolesnikov.
\newblock Mass transport and variants of the logarithmic {S}obolev inequality.
\newblock {\em J. Geom. Anal.}, 18(4):921--979, 2008.

\bibitem{beck_friedland}
Thomas Beck and David Jerison.
\newblock The {F}riedland-{H}ayman inequality and {C}affarelli's contraction
  theorem.
\newblock {\em J. Math. Phys.}, 62(10):Paper No. 101504, 11, 2021.

\bibitem{Bis81}
Jean-Michel Bismut.
\newblock Martingales, the {M}alliavin calculus and hypoellipticity under
  general {H}\"{o}rmander's conditions.
\newblock {\em Z. Wahrsch. Verw. Gebiete}, 56(4):469--505, 1981.

\bibitem{Bob97}
Sergey~G. Bobkov.
\newblock An isoperimetric inequality on the discrete cube, and an elementary
  proof of the isoperimetric inequality in {G}auss space.
\newblock {\em Ann. Probab.}, 25(1):206--214, 1997.

\bibitem{BTHD21}
Valentin~De Bortoli, James Thornton, Jeremy Heng, and Arnaud Doucet.
\newblock Diffusion {S}chr{\"o}dinger bridge with applications to score-based
  generative modeling.
\newblock In A.~Beygelzimer, Y.~Dauphin, P.~Liang, and J.~Wortman Vaughan,
  editors, {\em Advances in Neural Information Processing Systems}, 2021.

\bibitem{caffarelli2000monotonicity}
Luis~A. Caffarelli.
\newblock Monotonicity properties of optimal transportation and the {FKG} and
  related inequalities.
\newblock {\em Comm. Math. Phys.}, 214(3):547--563, 2000.

\bibitem{cheng2021}
Li-Juan Cheng, Feng-Yu Wang, and Anton Thalmaier.
\newblock Some inequalities on {R}iemannian manifolds linking entropy, {F}isher
  information, {S}tein discrepancy and {W}asserstein distance.
\newblock {\em arXiv preprint arXiv:2108.12755}, 2021.

\bibitem{CP}
Sinho Chewi and Aram-Alexandre Pooladian.
\newblock An entropic generalization of {C}affarelli's contraction theorem via
  covariance inequalities.
\newblock {\em arXiv preprint arXiv:2203.04954}, 2022.

\bibitem{CMP20}
Andrea Cianchi, V\'{\i}t Musil, and Lubo\v{s} Pick.
\newblock Moser inequalities in {G}auss space.
\newblock {\em Math. Ann.}, 377(3-4):1265--1312, 2020.

\bibitem{colombo_bounds}
Maria Colombo and Max Fathi.
\newblock Bounds on optimal transport maps onto log-concave measures.
\newblock {\em J. Differential Equations}, 271:1007--1022, 2021.

\bibitem{colombo_lipschitz}
Maria Colombo, Alessio Figalli, and Yash Jhaveri.
\newblock Lipschitz changes of variables between perturbations of log-concave
  measures.
\newblock {\em Ann. Sc. Norm. Super. Pisa Cl. Sci. (5)}, 17(4):1491--1519,
  2017.

\bibitem{cordero2002some}
Dario Cordero-Erausquin.
\newblock Some applications of mass transport to {G}aussian-type inequalities.
\newblock {\em Arch. Ration. Mech. Anal.}, 161(3):257--269, 2002.

\bibitem{ElworthyLi}
K.~David Elworthy and Xue-Mei Li.
\newblock Formulae for the derivatives of heat semigroups.
\newblock {\em J. Funct. Anal.}, 125(1):252--286, 1994.

\bibitem{FSX19b}
Xiao Fang, Qi-Man Shao, and Lihu Xu.
\newblock Correction to: {M}ultivariate approximations in {W}asserstein
  distance by {S}tein's method and {B}ismut's formula.
\newblock {\em Probab. Theory Related Fields}, 175(3-4):1177--1181, 2019.

\bibitem{FSX19a}
Xiao Fang, Qi-Man Shao, and Lihu Xu.
\newblock Multivariate approximations in {W}asserstein distance by {S}tein's
  method and {B}ismut's formula.
\newblock {\em Probab. Theory Related Fields}, 174(3-4):945--979, 2019.

\bibitem{FGP}
Max Fathi, Nathael Gozlan, and Maxime Prod'homme.
\newblock A proof of the {C}affarelli contraction theorem via entropic
  regularization.
\newblock {\em Calc. Var. Partial Differential Equations}, 59(3):Paper No. 96,
  18, 2020.

\bibitem{MR2648684}
Alessio Figalli.
\newblock Regularity of optimal transport maps [after {M}a-{T}rudinger-{W}ang
  and {L}oeper].
\newblock Number 332, pages Exp. No. 1009, ix, 341--368. 2010.
\newblock S\'{e}minaire Bourbaki. Volume 2008/2009. Expos\'{e}s 997--1011.

\bibitem{GDVM19}
Jackson Gorham, Andrew~B. Duncan, Sebastian~J. Vollmer, and Lester Mackey.
\newblock Measuring sample quality with diffusions.
\newblock {\em Ann. Appl. Probab.}, 29(5):2884--2928, 2019.

\bibitem{Har99}
Gilles Harg\'{e}.
\newblock A particular case of correlation inequality for the {G}aussian
  measure.
\newblock {\em Ann. Probab.}, 27(4):1939--1951, 1999.

\bibitem{HJA20}
Jonathan Ho, Ajay Jain, and Pieter Abbeel.
\newblock Denoising diffusion probabilistic models.
\newblock In H.~Larochelle, M.~Ranzato, R.~Hadsell, M.F. Balcan, and H.~Lin,
  editors, {\em Advances in Neural Information Processing Systems}, volume~33,
  pages 6840--6851. Curran Associates, Inc., 2020.

\bibitem{Hsu}
Elton~P. Hsu.
\newblock {\em Stochastic analysis on manifolds}, volume~38 of {\em Graduate
  Studies in Mathematics}.
\newblock American Mathematical Society, Providence, RI, 2002.

\bibitem{IvRu20}
Paata Ivanisvili and Ryan Russell.
\newblock Exponential integrability for log-concave measures.
\newblock {\em arXiv preprint arXiv:2004.09704}, 2020.

\bibitem{KaratzasShreve}
Ioannis Karatzas and Steven~E. Shreve.
\newblock {\em Brownian motion and stochastic calculus}, volume 113 of {\em
  Graduate Texts in Mathematics}.
\newblock Springer-Verlag, New York, second edition, 1991.

\bibitem{KiMi12}
Young-Heon Kim and Emanuel Milman.
\newblock A generalization of {C}affarelli's contraction theorem via (reverse)
  heat flow.
\newblock {\em Math. Ann.}, 354(3):827--862, 2012.

\bibitem{klartag_Putterman}
Bo’az Klartag and Eli Putterman.
\newblock Spectral monotonicity under gaussian convolution.
\newblock {\em Ann. Fac. Sci. Toulouse Math}, To appear.

\bibitem{kolesnikov2011mass}
Alexander~V Kolesnikov.
\newblock Mass transportation and contractions.
\newblock {\em arXiv preprint arXiv:1103.1479}, 2011.

\bibitem{kolesnikov_sobolev}
Alexander~V. Kolesnikov.
\newblock On {S}obolev regularity of mass transport and transportation
  inequalities.
\newblock {\em Theory Probab. Appl.}, 57(2):243--264, 2013.

\bibitem{LaSa22}
Hugo Lavenant and Filippo Santambrogio.
\newblock The flow map of the {F}okker-{P}lanck equation does not provide
  optimal transport.
\newblock {\em Appl. Math. Lett.}, 133:Paper No. 108225, 7, 2022.

\bibitem{mikulincer_clt}
Dan Mikulincer.
\newblock A {CLT} in {S}tein's distance for generalized {W}ishart matrices and
  higher-order tensors.
\newblock {\em Int. Math. Res. Not. IMRN}, (10):7839--7872, 2022.

\bibitem{MS21}
Dan Mikulincer and Yair Shenfeld.
\newblock On the {L}ipschitz properties of transportation along heat flows.
\newblock {\em GAFA Seminar Notes}, To appear.

\bibitem{milman2018spectral}
Emanuel Milman.
\newblock Spectral estimates, contractions and hypercontractivity.
\newblock {\em J. Spectr. Theory}, 8(2):669--714, 2018.

\bibitem{milman2020reverse}
Emanuel Milman.
\newblock Reverse {H}\"{o}lder inequalities for log-{L}ipschitz functions.
\newblock {\em Pure Appl. Funct. Anal.}, 8(1):297--310, 2023.

\bibitem{neeman2022lipschitz}
Joe Neeman.
\newblock Lipschitz changes of variables via heat flow.
\newblock {\em arXiv preprint arXiv:2201.03403}, 2022.

\bibitem{otto2000generalization}
F.~Otto and C.~Villani.
\newblock Generalization of an inequality by {T}alagrand and links with the
  logarithmic {S}obolev inequality.
\newblock {\em J. Funct. Anal.}, 173(2):361--400, 2000.

\bibitem{shenfeld2022}
Yair Shenfeld.
\newblock Exact renormalization groups and transportation of measures.
\newblock {\em arXiv preprint arXiv:2205.01642}, 2022.

\bibitem{tanana_comparison}
Anastasiya Tanana.
\newblock Comparison of transport map generated by heat flow interpolation and
  the optimal transport {B}renier map.
\newblock {\em Commun. Contemp. Math.}, 23(6):Paper No. 2050025, 7, 2021.

\bibitem{thompson2020}
James Thompson.
\newblock Approximation of {R}iemannian measures by {S}tein's method.
\newblock {\em arXiv preprint arXiv:2001.09910}, 2020.

\bibitem{villani2008optimal}
C{\'e}dric Villani.
\newblock {\em Optimal transport: old and new}, volume 338.
\newblock Springer Science \& Business Media, 2008.

\end{thebibliography}
\end{document}